\documentclass[12pt]{article}
\usepackage{amssymb}
\usepackage{amsmath}
\usepackage{amsthm}
\usepackage{fixltx2e}
\usepackage{color}
\usepackage{mnsymbol}

\newtheorem{thm}{Theorem}
\newtheorem{cor}[thm]{Corollary}
\newenvironment{rem}[1][Remark:]{\begin{trivlist}
\item[\hskip \labelsep {\bfseries #1}]}{\end{trivlist}}
\newtheorem{prop}[thm]{Proposition}

\theoremstyle{definition}
\newtheorem{defi}[thm]{Definition}

\numberwithin{equation}{section}
\numberwithin{thm}{section}

\newcommand{\rn}[1]{{\mathbb{R}^{#1}}}

\newcommand{\ep}{\epsilon}

\newcommand{\la}{\langle}
\newcommand{\ra}{\rangle}
\newcommand{\lla}{\llangle}
\newcommand{\rra}{\rrangle}

\newcommand{\gr}{\nabla}
\newcommand{\cd}{\cdot}
\newcommand{\dt}{\partial_t}

\newcommand{\alg}{\mathfrak{g}}
\newcommand{\dalg}{\mathfrak{g}^*}

\newcommand{\re}{\mathrm{Re}}
\newcommand{\im}{\mathrm{Im}}

\newcommand{\mad}{\mathrm{\mathbf{M}}}									

\newcommand{\xip}{\xi_\Psi}

\newcommand{\pbfl}[2]{\left\{ {#1},{#2} \right\}_\mathrm{CF}}		
\newcommand{\pbnls}[2]{\left\{ {#1},{#2} \right\}_\mathrm{NLS}}		

\newcommand{\fl}{\mathrm{CF}}
\newcommand{\nls}{\mathrm{NLS}}

\newcommand{\dfm}{\dfrac{\delta F}{\delta \mu}}	
\newcommand{\dgm}{\dfrac{\delta G}{\delta \mu}}
\newcommand{\dfr}{\dfrac{\delta F}{\delta \rho}}
\newcommand{\dgr}{\dfrac{\delta G}{\delta \rho}}

\newcommand{\psib}{\bar{\psi}}
\newcommand{\phib}{\bar{\phi}}

\newcommand{\salg}{\mathfrak{s}}
\newcommand{\dsalg}{\mathfrak{s}^*}

\newcommand{\bfj}{\mathrm{\mathbf{J}}}

\newcommand{\Apsi}{\mathcal{A}_\Psi}
\newcommand{\Asd}{\mathcal{A}_{\mathfrak{s}^*}}

\begin{document}

\title{The Madelung Transform as a Momentum Map}
\author{Daniel Fusca\footnote{Department of Mathematics, University of Toronto, Toronto, ON M5S 2E4, Canada; e-mail:
dfusca@math.toronto.edu}}
\date{}

\maketitle

\begin{abstract}
The Madelung transform relates the non-linear Schr\"{o}dinger equation and a compressible Euler equation known as the quantum hydrodynamical system. We prove that the Madelung transform is a momentum map associated with an action of the semidirect product group $\mathrm{Diff}(\mathbb{R}^{n})\ltimes H^\infty(\mathbb{R}^{n})$, which is the configuration space of compressible fluids, on the space $\Psi = H^\infty(\mathbb{R}^{n}, \mathbb{C})$ of wave functions. In particular, this implies that the Madelung transform is a Poisson map taking the natural Poisson bracket on $\Psi$ to the compressible fluid Poisson bracket. Moreover, the Madelung transform provides an example of ``Clebsch variables'' for the hydrodynamical system.
\end{abstract}

\begin{section}{Introduction}

The Madelung transform was introduced in 1927, soon after the birth of quantum mechanics, as a way to relate Schr\"{o}dinger-type equations to hydrodynamical equations \cite{Madelung1927}.  It turns out that the Madelung transform not only maps one equation to the other, but it also preserves the Hamiltonian properties of both equations. Namely, the non-linear Schr\"{o}dinger equation is Hamiltonian with respect to the constant Poisson structure on the space of wave functions, which are complex valued fast decaying smooth functions on $\rn{n}$. On the other hand, the compressible Euler equation is Hamiltonian with respect to the natural Lie-Poisson structure on the space of pairs consisting of fluid momenta $\mu$ and fluid densities $\rho$. This space is the dual of the Lie algebra of the semidirect product group of the group of diffeomorphisms of $\rn{n}$ times the space of real-valued fast decaying functions, which is the configuration space of a compressible fluid.  In this paper we show that the Madelung transform sends one Poisson structure to the other. Moreover, the transform is a momentum map associated with a natural action of this semidirect product group on the space of wave functions.

Let complex valued functions $\psi \in \Psi = H^\infty(\rn{n};\mathbb{C}) := \cap_{k\geq 0} H^k(\rn{n};\mathbb{C})$ evolve according to the non-linear Schr\"{o}dinger equation
\begin{equation}\label{eq:nls} 
	\dt \psi = \frac{i}{2}\left(\Delta \psi - 2f\left(| \psi |^2 \right)\psi \right).  
	\tag{NLS}
\end{equation}
Here $f:\rn{} \to \rn{}$ is a smooth function which characterizes the type of non-linear Schr\"{o}dinger equation being considered. For example, the Gross-Pitaevskii equation corresponds to $f(r)=r-1$. Depending on the type of non-linearity, one may consider $\psi$ belonging to an appropriate function space other than $H^\infty(\rn{n};\mathbb{C})$, but for simplicity this will not be discussed in the present paper, see \cite{Carles2012} for details.

The Madelung transform is a map between the NLS-type and hydrodynamical equations. It takes a nonvanishing complex valued function $\psi$ to the pair of real valued functions $(\tau, \rho)$ defined by $\psi = \sqrt{\rho}e^{i\tau}$.
Such a substitution for $\psi$ sends \eqref{eq:nls} into a system of equations for the functions $\rho$ and $v := \gr \tau$ known as the ``quantum hydrodynamical system''
\begin{equation}\label{eq:qhd1}
\left\{
  \begin{array}{l}
   \dt \rho = -\gr \cd (\rho v) \\
   \dt v = -v \cd \gr v - \gr \left( f (\rho) - \dfrac{\Delta \left(\sqrt{\rho}\right)}{2\sqrt{\rho}} \right)
  \end{array} \right. . 
  \tag{QHDa}
\end{equation}
The first equation of this system is the continuity equation for a density $\rho$ moved by a flow with velocity $v$. The second equation would be the classical Euler equation of a barotropic fluid except for the fact that the ``quantum pressure'' $\mathcal{P}(\rho) := \frac{\Delta\sqrt{\rho}}{2\sqrt{\rho}}$ depends on both $\rho$ and its derivatives rather than just on $\rho$ itself.

This system can be written in terms of the momentum, which is the 1-form $\mu = \rho v^\flat$ defined with respect to the Euclidean metric on $\rn{n}$. (Throughout this paper $v^\flat$ is identified with $v$, so we write $\mu = \rho v$ as well.) Assuming $\rho$ is always positive, the system \eqref{eq:qhd1} is equivalent to the following:
\begin{equation}\label{eq:qhd}
\left\{
  \begin{array}{l}
   \dt \rho = -\gr \cd \mu \\
   \dt \mu = -\gr \cd \left( \dfrac{1}{\rho} \mu \otimes \mu \right) - \rho \gr \left( f (\rho) - \dfrac{\Delta \left(\sqrt{\rho}\right)}{2\sqrt{\rho}} \right)
  \end{array} \right. . 
  \tag{QHD}
\end{equation}
The term $\gr \cd \left( \frac{1}{\rho} \mu \otimes \mu \right)$, in components, is given by
\[
	\left( \gr \cd \left( \frac{1}{\rho} \mu \otimes \mu\right) \right)_j = \sum_i \partial_i (\frac{1}{\rho}\mu_i \mu_j).
\]

Note that $\rho$ and $\mu$ are natural coordinates on the dual of the Lie algebra of the semidirect Lie group $S = \mathrm{Diff}(\mathbb{R}^{n})\ltimes H^\infty(\mathbb{R}^{n})$ which is the configuration space of the compressible fluid. Here $\mathrm{Diff}(\mathbb{R}^{n})$ stands for the group of diffeomorphisms that asymptotically approach the identity map at infinity. 

\begin{rem}
	While we formulate and prove the main result in the setting of $\rn{n}$, the definitions and all proofs can be extended to an arbitrary manifold with volume form $d\mathrm{Vol}$ by replacing all gradients with exterior differentiation and by defining the divergence of a vector field in terms of $d\mathrm{Vol}$. Note however that one needs a metric structure to define the equations \eqref{eq:nls} and \eqref{eq:qhd} on manifolds.
\end{rem}

We start the paper with a review of the relevant geometric structures associated with \eqref{eq:nls} and \eqref{eq:qhd} to
establish the context  for the Madelung transform. The Hamiltonian structures of these systems are reviewed in 
Section~\ref{sec:hamstruc}, while Section~\ref{sec:groupaction} defines the action of the semidirect product group $S$ on the space of wave functions $\Psi$. The proof of the main result, that the Madelung transform is a momentum map for this action, is given in Section~\ref{sec:madelung}, Theorem~\ref{thm:madmoment}. For more details on applications of the Madelung transform we refer to \cite{Carles2012}.

\end{section}

\begin{section}{Geometric preliminaries}
\label{sec:geometry}

\begin{subsection}{Hamiltonian structures of non-linear Schr\"{o}dinger and the quantum hydrodynamical system}\label{sec:hamstruc}

Both \eqref{eq:nls} and \eqref{eq:qhd} are Hamiltonian systems with respect to the following Hamiltonians and Poisson structures. Note that these Poisson structures are only defined a subclasses $\Apsi \subset C^\infty(\Psi)$ and $\Asd \subset C^\infty(\dsalg)$ of smooth functionals on $\Psi$ and $\dsalg$. As always with Poisson brackets on infinite-dimensional spaces, the definition of the Poisson algebra of functionals is a subtle question. For instance, functionals on a space of functions $u$ defined as integrals of polynomials (in $u$ and finitely many derivatives of $u$) are closed under the Poisson bracket but not under multiplication.  The same is true of functionals having smooth $L^2$ gradients: they are closed under the Poisson bracket but not under multiplication, cf. \cite{kolev2007}.

\vspace{2 mm}

We start with \eqref{eq:nls} and consider the space $\Psi = H^\infty(\rn{n};\mathbb{C})$ of complex valued functions $\psi$. The real Hermitian inner product on $\Psi$ is defined by $\la f, g \ra := \re \int \bar{f}g \, dx$, and the gradient $\gr$ is defined with respect to this inner product. The Poisson bracket on $\Psi$ is given by
\begin{equation*}\label{eq:pbnls}
	\pbnls{F}{G}(\psi) = \left \la \gr F ,-\tfrac{i}{2} \gr G \right \ra .
\end{equation*}
The Hamiltonian associated with \eqref{eq:nls} is
\begin{equation*}\label{eq:hamnls}
	H_\nls = \int_\rn{n} \frac{1}{2} | \gr \psi |^2 + U(|\psi|^2) dx,
\end{equation*}
where $U:\rn{}\to\rn{}$ is a function satisfying $U'=f$. One finds that the Hamiltonian vector field $X_{H_\nls}$ associated with this Hamiltonian functional and Poisson bracket $\pbnls{\cd}{\cd}$ is given by 
\[
	X_{H_\nls} = -\frac{i}{2} \gr H = \frac{i}{2}(\Delta \psi - 2f(| \psi |^2)\psi ),
\]
which is the right-hand side of \eqref{eq:nls}.

\vspace{2 mm}

Now consider the equation \eqref{eq:qhd}, which describes the motion of a compressible isentropic-type fluid. The Poisson geometry of such fluids was studied in \cite{MarsRaWein84}, and we outline the results below.

Consider the semidirect product group $S=\mathrm{Diff}(\rn{n})\ltimes H^\infty(\rn{n}; \rn{})$. Here and below we assume that the elements ``decay sufficiently fast at infinity." The space $H^\infty(\rn{n}; \rn{})$ is defined as $H^\infty(\rn{n}; \rn{}) := \cap_{k\geq 0} H^k(\rn{n}; \rn{})$, while $\mathrm{Diff}(\rn{n})$ is the set of diffeomorphisms on $\rn{n}$ of the form $g = \mathrm{id} + f$ with $f \in H^\infty(\rn{n}; \rn{n})$. The Lie algebra of the group $S$ is $\salg=\mathrm{vect}(\rn{n})\ltimes H^\infty(\rn{n}; \rn{})$, and the (smooth) dual of this Lie algebra is $\dsalg=\mathrm{vect}^*(\rn{n})\oplus H^{\infty*}(\rn{n}; \rn{})$. The space $\dsalg$ is the space of elements $(\mu,\rho)$, where $\rho$ is a density and $\mu=\rho v^\flat$ is a 1-form defined with respect to the Euclidean metric on $\rn{n}$. As mentioned above, we also identify $\mu$ with $\rho v$. There is a natural linear Poisson structure on the dual of any Lie algebra, the Lie-Poisson bracket. For this reason we work with the momentum $\mu$ rather than the velocity $v$. The Lie-Poisson bracket on $\dsalg$ is given by
\begin{align*}\label{eq:pbfl}
	\pbfl{F}{G}(\mu , \rho) &= \int_\rn{n} \mu \cd \left[ \left(\dgm \cd \gr \right) \dfm - \left( \dfm \cd \gr \right) \dgm \right] dx \\
	& \quad + \int_\rn{n} \rho \left[ \left( \dgm \cd \gr \right) \dfr - \left( \dfm \cd \gr \right) \dgr \right] dx .
\end{align*}
This bracket is also called the \textit{compressible fluid bracket}. The Hamiltonian for \eqref{eq:qhd}, written in terms of momentum $\mu$, is
\begin{equation*}\label{eq:hamqhd}
	H_\fl = \int_\rn{n} \frac{1}{2} \frac{| \mu |^2}{\rho} + \frac{1}{8}\frac{\left| \gr \rho \right|^2}{\rho} + U(\rho) \, dx.
\end{equation*}
We denote the associated Hamiltonian vector field on the $(\mu,\rho)$-space by $X_{H_\fl}$. One can check that it agrees with the right-hand side of \eqref{eq:qhd}: if $X^\rho_{H_\fl}$ and $X^\mu_{H_\fl}$ denote the $\rho$ and $\mu$ components of $X_{H_\fl}$, then
\begin{align*}
	X^\rho_{H_\fl} &= -\gr \cd \mu \, ,\\
	X^\mu_{H_\fl} &= -\gr \cd \left( \dfrac{1}{\rho} \mu \otimes \mu \right) - \rho \gr \left( f (\rho) - \dfrac{\Delta \left(\sqrt{\rho}\right)}{2\sqrt{\rho}} \right).
\end{align*}

\vspace{2 mm}

\end{subsection}

\begin{subsection}{Lie group and Lie algebra actions on the space of wave functions}\label{sec:groupaction}
It turns out that it is natural to think of $\Psi$ as being a space of complex valued half-densities on $\rn{n}$
since $\psi$ is square-integrable and $|\psi|^2$ is often interpreted as a probability measure.
Half-densities are characterized by how they are transformed under diffeomorphisms of the underlying space: if $\psi$ is a half-density on $\rn{n}$ and $g$ is a diffeomorphism of $\rn{n}$, the pushforward $g_*(\psi)$ of $\psi$ is $g_*(\psi) = \sqrt{|\mathrm{Det}(Dg^{-1})|}\, \psi \circ g^{-1}$. With this in mind, the following action of $S$ on $\Psi$ is natural.

\begin{defi}
	The semidirect product group $S=\mathrm{Diff}(\rn{n})\ltimes H^\infty(\rn{n})$ {\it acts on the Poisson space} $(\Psi,\pbnls{\cd}{\cd})$ as follows: if $(g,a)\in S$ is a group element, then
	\begin{equation}\label{eq:grpact}
		(g,a): \psi \mapsto (g,a) \cd \psi := \sqrt{|\mathrm{Det}(Dg^{-1})|}\, e^{-ia} (\psi \circ g^{-1}).
	\end{equation}
\end{defi}
\vspace{2 mm}
In other words, $\psi$ is pushed forward under the diffeomorphism $g$ as a complex-valued half-density, followed by a pointwise phase adjustment $e^{-ia}$.

For the above Lie group action \eqref{eq:grpact}, the Lie algebra action is as follows.
\begin{prop}
	Given an element $\xi = (v,\alpha)\in \salg = \mathrm{vect}(\rn{n})\ltimes H^\infty(\rn{n})$, the infinitesimal action of $\xi$ on $\Psi$ corresponding to the action \eqref{eq:grpact} is the vector field $\xip \in \mathfrak{X}(\Psi)$ defined at each point $\psi$ by
	\begin{equation}\label{eq:infact}
		\xip(\psi) = -\frac{1}{2} \psi \gr \cd v - i\alpha \psi - \gr \psi \cd v.
	\end{equation}
\end{prop}
\begin{proof}
	The proof is a direct computation.
\end{proof}

In the next section we define the Madelung transform and show that it is a momentum map associated with the action \eqref{eq:infact}.

\end{subsection}

\end{section}

\begin{section}{The Madelung transform and a geometric interpretation}

\begin{subsection}{The Madelung transform}

The classical Madelung transform is a map $(\tau,\rho) \mapsto \psi = \sqrt{\rho} e^{i\tau}$ defined for positive $\rho$. We consider the inverse map as a more fundamental object and define it below.

\begin{defi}
	The (inverse) {\it Madelung transform} is the map $\mad : \Psi \to \dsalg$ defined by
	\begin{equation}\label{eq:mad}
		\mad(\psi) := \left( \begin{array}{c}
						\im \, \psib \gr \psi  \notag \\ \re \, \psib \psi \end{array} \right)
					= \left( \begin{array}{c}
						\mu \notag \\ \rho \end{array} \right).
	\end{equation}
\end{defi}

\begin{prop}
	The map $\mad$ is the inverse of the classical Madelung transform in the sense that if $\psi = \sqrt{\rho}e^{i\tau}$, then $\mad(\psi) = (\rho \gr \tau , \rho)$. If $\rho$ is positive, then $(\rho \gr \tau , \rho)$ can be identified with $([\tau],\rho)$, where $[\cd]$ is the equivalence class of functions on $\rn{n}$ modulo an additive constant.
\end{prop}

\begin{proof}
	By the definition of $\mad$,
	\begin{align*}
		\mad(\sqrt{\rho}e^{i\tau}) &= \left( \begin{array}{c}
						\im \, \sqrt{\rho}e^{-i\tau} \gr \left( \sqrt{\rho}e^{i\tau} \right) \\ \re \, \sqrt{\rho}e^{-i\tau} \sqrt{\rho}e^{i\tau} \end{array} \right) \\
					&= \left( \begin{array}{c}
						\im \, \sqrt{\rho}e^{-i\tau} \left( \tfrac{1}{2} \rho^{-\frac{1}{2}} e^{i\tau} \gr \rho + i \sqrt{\rho} e^{i\tau} \gr \tau \right) \\ \rho \end{array} \right) = \left( \begin{array}{c}
						\rho \gr \tau \\ \rho \end{array} \right).
	\end{align*}
	If $\rho$ is positive, one can recover $[\tau]$ from $\rho \gr \tau$ by dividing by $\rho$ and integrating.
\end{proof}

\begin{rem}
	Note that the equivalence of functions $\tau$ and $\tau'$ differing by an additive constant corresponds to the physical equivalence of two wave functions $\psi$ and $\psi'$ differing by a constant phase factor.
\end{rem}

\begin{prop}
	{\rm (cf. eg. \cite{Madelung1927})} The (inverse) Madelung transform $\mad$ sends \eqref{eq:nls} to \eqref{eq:qhd}:
	 $d\mad (X_{H_\nls}) = X_{H_\fl}$. 
\end{prop}

\begin{proof}
	The pushforward $d\mad_\psi (\phi)$ of a tangent vector $\phi \in T_\psi \Psi$ by the map $\mad$ is given by
	\begin{align*}
		d\mad_\psi (\phi) &= \left(\begin{array}{l} \im (\psib \gr \phi + \phib \gr \psi) \\ 2\re (\psib \phi) \end{array}\right) 
			=: \left(\begin{array}{l} \left(d\mad_\psi (\phi) \right)^\mu \\ \left(d\mad_\psi (\phi) \right)^\rho \end{array} \right).
	\end{align*}

	Recall that $X_{H_\nls} = \frac{i}{2}(\Delta \psi - 2f(| \psi |^2)\psi)$. Now substitute $\psi = \sqrt{\rho} e^{i\tau}$ into the expression for $X_{H_\nls}$, which results in
	\begin{align} \label{eq:pushcomp1}
		X_{H_\nls} &= \frac{i}{2} \left(\Delta \left(\sqrt{\rho} e^{i\tau} \right) -2f(\rho)\sqrt{\rho} e^{i\tau} \right) \notag \\
			&= \frac{i}{2} \left( -\frac{|\gr \rho|^2}{4\rho^{\frac{3}{2}}} e^{i\tau} + \frac{\Delta \rho}{2\sqrt{\rho}} e^{i\tau} - \sqrt{\rho} |\gr \tau|^2 e^{i\tau} - 2f(\rho)\sqrt{\rho}e^{i\tau} \right) \notag \\
			&\qquad - \frac{1}{2} \left( \frac{\gr \rho \cd \gr \tau}{\sqrt{\rho}} e^{i \tau} + \sqrt{\rho} \Delta \tau e^{i\tau} \right) .
	\end{align}

	The $\rho$-component of the image $d\mad_\psi (X_{H_\nls})$ is obtained by multiplying \eqref{eq:pushcomp1} by $2\psib = 2\sqrt{\rho}e^{-i\tau}$ and then taking the real part. We have 
	\begin{align*}
		\left(d\mad_\psi (X_{H_\nls}) \right)^\rho &= 2\re (\psib X_{H_\nls}) = -\left(\gr \rho \cd \gr \tau + \rho \Delta \tau \right) \\
			&= -\gr \cd \left( \rho \gr \tau \right) = -\gr \cd \mu = X_{H_\fl}^\rho \, ,
	\end{align*}
	which is the right-hand side of the continuity equation in \eqref{eq:qhd}.

	The $\mu$-component of $d\mad_\psi (X_{H_\nls})$ is found in a similar fashion. Namely, after straightforward computations, one obtains
	\begin{align*}
		\left(d\mad_\psi (X_{H_\nls}) \right)^\mu &= \im \left( \psib \gr X_{H_\nls} + \bar{X}_{H_\nls} \gr \psi \right) \\
			&= \frac{|\gr \rho|^2 \gr \rho}{4\rho^2} - \frac{\gr |\gr \rho|^2}{8\rho} - \frac{\Delta \rho \gr \rho}{4\rho} + \frac{\gr \Delta \rho}{4} \\
			&\qquad - \frac{\rho\gr|\gr \tau|^2}{2} - \gr \tau \gr \cd (\rho \gr \tau) - \rho \gr f(\rho) .
	\end{align*}
	Direct computation shows that the first line of the right-hand side of the last equality is equal to $\rho \gr \left( \frac{\Delta \sqrt{\rho}}{2\sqrt{\rho}}\right)$. The terms involving $\tau$ simplify to
	\begin{align*}
		- \frac{\rho\gr|\gr \tau|^2}{2} - \gr \tau \gr \cd (\rho \gr \tau) &= -\left[ (\rho \gr \tau \cd \gr) \gr \tau + \gr \tau \gr \cd (\rho \gr \tau) \right] \\ 
			&= -\left[(\mu \cd \gr)\left(\frac{\mu}{\rho}\right) + \frac{\mu}{\rho} \gr \cd \mu \right] 
			= - \gr \cd \left( \frac{1}{\rho} \mu \otimes \mu \right),
	\end{align*}
	so that 
	\[
		\left(d\mad_\psi (X_{H_\nls}) \right)^\mu = - \gr \cd \left( \frac{1}{\rho} \mu \otimes \mu \right) - \rho \gr \left( f(\rho) - \frac{\Delta \sqrt{\rho}}{2\sqrt{\rho}} \right) = X_{H_\fl}^\mu \, .
	\]	
\end{proof}

\end{subsection}

\begin{subsection}{The Madelung transform as a momentum map}\label{sec:madelung}
We first recall the definition of a momentum map.

Suppose we are given a Poisson manifold $P$, a Lie algebra $\alg$, and an action $A:\alg \to \mathfrak{X}(P)$, $A(\xi)=\xi_P$. Let $\lla , \rra$ denote the pairing of $\alg$ and $\dalg$. The Lie algebra action $A$ \textit{admits a momentum map} if there exists a map $\bfj :P \to \dalg$ satisfying the following definition:

\begin{defi}\label{def:mommap}
	A \textit{momentum map} associated with a Lie algebra action $A(\xi)=\xi_P$ is a map $\bfj :P \to \dalg$ such that for every $\xi \in \alg$ the function $J(\xi) : P \to \rn{}$ defined by $J(\xi)(p) := \lla \bfj (p), \xi \rra$ satisfies
	\begin{equation*}\label{mommapdef}
		X_{J(\xi)} = \xi_P.
	\end{equation*}
\end{defi}
\vspace{2 mm}
Thus Lie algebra actions that admit momentum maps are Hamiltonian actions on $P$, and the pairing of the momentum map at a point with an element $\xi \in \alg$ returns a Hamiltonian function associated with the Hamiltonian vector field $\xi_P$. We now show that $\mad$ is a momentum map associated with the action \eqref{eq:infact}. In the proofs below we will use the notation of Definition~\ref{def:mommap} and define $M(\xi) : \Psi \to \rn{}$ by $M(\xi)(\psi) := \lla \mad (\psi) , \xi \rra$.

\begin{thm}\label{thm:madmoment}
	For the Lie algebra $\salg = \mathrm{vect}(\rn{n})\ltimes H^\infty(\rn{n})$, its action \eqref{eq:infact} on the Poisson space $\Psi = H^\infty(\rn{n};\mathbb{C})$ equipped with the Poisson structure $\pbnls{F}{G}(\psi) = \left \la \gr F ,-\tfrac{i}{2} \gr G \right \ra$ admits a momentum map. The map $\mad : \Psi \to \dsalg$ defined by \eqref{eq:mad} is a momentum map associated with this Lie algebra action.
\end{thm}

\begin{proof}
	The Hamiltonian vector field of $M(\xi)$ is given by
	\[
		X_{M(\xi)} = -\tfrac{i}{2} \gr M(\xi),
	\]
	where the gradient is defined with respect to the inner product $\la f , g \ra = \re \int \bar{f} g \, dx $. 
Let $\xi = (v,\alpha)$ be an element of $\salg = \mathrm{vect}(\rn{n})\ltimes H^\infty(\rn{n})$ and its pairing with
$(\mu, \rho)\in \dsalg$ is given by $\lla (v,\alpha), (\mu, \rho)\rra:=\int_{\rn{n}} \rho \cd \alpha + \mu \cd v \, dx$.
		 We have
	\begin{align*}
		M(\xi)(\psi) &= \int_{\rn{n}} \mad (\psi)^\rho\cdot \alpha + \mad (\psi)^\mu \cd v \, dx \\
			&= \re \int_{\rn{n}} \psib \psi \alpha - i \psib \gr \psi \cd v \, dx \ .
	\end{align*}
	To find the gradient, or variational derivative, let $\phi \in C_c^\infty(\rn{n};\mathbb{C})$ be a test function and consider the variation of $\psi$ in the direction $\phi$:
	\begin{align*}
		\left. \frac{d}{d\ep} M(\xi)(\psi + \ep \phi)\right|_{\ep = 0} &= \re \int_{\rn{n}} \psib \phi \alpha + \phib \psi \alpha - i \phib \gr \psi \cd v - i\psib \gr \phi \cd v \, dx \\
			&= \re \int_{\rn{n}} 2 \phib \psi \alpha - i \phib \gr \psi \cd v + i \phi \gr \cd (\psib v) \, dx \\
			&= \re \int_{\rn{n}} \phib\left[ 2 \psi \alpha - 2 i \gr \psi \cd v - i \psi \gr \cd v \right] \, dx\, ,
	\end{align*}
	so that $\gr M(\xi)(\psi) = 2 \psi \alpha - 2i \gr \psi \cd v - i \psi \gr \cd v$. Finally we conclude that
	\begin{align*}
		X_{M(\xi)}(\psi) = -i\alpha \psi - \gr \psi \cd v - \frac{1}{2} \psi \gr \cd v.
	\end{align*}
	Comparing this with \eqref{eq:infact}, one obtains that $X_{M(\xi)}(\psi) = \xip(\psi)$.
\end{proof}

Any momentum map $\bfj: P \to \dalg$ is also a Poisson map taking the bracket on $P$ to the Lie-Poisson bracket on $\dalg$ provided that it is \textit{infinitesimally equivariant} (see, for example, \cite[Thm.~12.4.1]{Marsden2002}). Recall that a momentum map $\bfj$ of a Lie algebra $\alg$ is infinitesimally equivariant if for all $\xi,\,\eta \in \alg$ the following holds:
\begin{equation}\label{eq:infeq}
	J([\xi , \eta ]) = \{ J(\xi) , J(\eta)\}.
\end{equation}

\begin{thm}\label{thm:infeq}
	The map $\mad : \Psi \to \dsalg$ is infinitesimally equivariant for the action of the semidirect product Lie algebra $\salg$.
\end{thm}

\begin{proof}
	We start by computing the left hand side of \eqref{eq:infeq}. In our context, the Lie bracket on a pair of elements $\xi = (u,\alpha),\, \eta = (v,\beta) \in \salg = \mathrm{vect}(\rn{n})\ltimes H^\infty(\rn{n})$ is given by
	\begin{equation}\label{eq:algbrack}
		\left[(u,\alpha),(v,\beta)\right] = \left([u,v], v \cd \gr \alpha - u \cd \gr \beta \right).
	\end{equation}
	We refer to \cite{MarsRaWein84} for the general formula of the Lie bracket on a semidirect product algebra.
	Using the bracket \eqref{eq:algbrack} and the definition of $\mad$, we have
	\begin{align*}
		M([\xi , \eta ])(\psi) &= \re \int_{\rn{n}} -i\psib \gr \psi \cd [u,v] + \psib \psi \left(v \cd \gr \alpha - u \cd \gr \beta \right) dx.
	\end{align*}
	For the right-hand side of \eqref{eq:infeq} we obtain
	\begin{align}\label{eq:step1}
		&\pbnls{M(\xi)}{M(\eta)}(\psi) = \la \gr M(\xi) , \tfrac{-i}{2} \gr M(\eta) \ra (\psi) \notag \\
			& \qquad = \frac{1}{2} \re \int_{\rn{n}} \left[ 2\psib \alpha + 2i\gr \psib \cd u + i \psib \gr \cd u \right] \notag \\
			& \qquad \qquad \qquad \times \left[ -2i\psi \beta -2\gr \psi \cd v - \psi \gr \cd v \right]dx  \notag \\
			& \qquad = \re \int_{\rn{n}} \left[-2\psib \gr \psi \cd v \alpha - \psib \psi \gr \cd v \alpha + 2\psi \gr \psib \cd u \beta + \psib \psi \gr \cd u \beta \right] \notag \\
			& \qquad \qquad \qquad + \left[ -i \gr \psib \cd u \gr \psi \cd v + i \gr \psi \cd u \gr \psib \cd v - i\psi \gr \psib \cd u \gr \cd v + i\psi \gr \psib \cd v \gr \cd u \right]dx \, .
	\end{align}
	At this point we use two identities that are easily verified. The first is
	\begin{equation}\label{eq:ident1}
		-2\psib \gr \psi \cd v \alpha - \psib \psi \gr \cd v \alpha = -\gr \cd \left(\psib \psi v \alpha \right) + \psib \psi v \cd \gr \alpha + \psi \gr \psib \cd v \alpha - \psib \gr \psi \cd v \alpha.
	\end{equation}
	Notice that $\psi \gr \psib \cd v \alpha - \psib \gr \psi \cd v \alpha$ is purely imaginary, so this term will not contribute to the integral in \eqref{eq:step1}. Neither will $-\gr \cd \left(\psib \psi v \alpha \right)$ since it is an exact derivative. There is a similar identity involving $u$ and $\beta$ in the place of $v$ and $\alpha$.
	
	The second identity we use is
	\begin{align}\label{eq:ident2}
		&-i\gr \psib \cd u \gr \psi \cd v + i \gr \psi \cd u \gr \psib \cd v - i \psi \gr \psib \cd u \gr \cd v + i\psi \gr \psib \cd v \gr \cd u \notag \\
			&\qquad \qquad \qquad = -i\gr \cd (\psi \gr \psib \cd u v) + i\gr \cd (\psi \gr \psib \cd v u) \notag \\
			&\qquad \qquad \qquad \qquad + i\psi \gr \left( \gr \psib \cd u \right)\cd v - i\psi \gr \left( \gr \psib \cd v \right)\cd u \notag \\
			&\qquad \qquad \qquad = -i\gr \cd (\psi \gr \psib \cd u v) + i\gr \cd (\psi \gr \psib \cd v u) - i\psi \gr \psib \cd [u,v].
	\end{align}
	The first two terms on the right-hand side of \eqref{eq:ident2} do not contribute to \eqref{eq:step1}.
	
	So, using \eqref{eq:ident1} and \eqref{eq:ident2}, we can rewrite the Poisson bracket in \eqref{eq:step1} as
	\begin{equation}
		\pbnls{M(\xi)}{M(\eta)}(\psi) =  \re \int_{\rn{n}} \psib \psi \left[ v \cd \gr \alpha - u \cd \gr \beta \right] - i\psib \gr \psi \cd [u,v] \, dx. \notag \\
	\end{equation}
	Since the latter expression is equal to $M([\xi , \eta ])(\psi)$, this completes the proof.
	
\end{proof}

\begin{cor}\label{poisson}
	The map $\mad$ is a Poisson map sending the bracket $\pbnls{\cd}{\cd}$ to the bracket $\pbfl{\cd}{\cd}$.
\end{cor}

\begin{proof}
	This is an immediate corollary of Theorems~\ref{thm:madmoment} and \ref{thm:infeq}.
\end{proof}

One can also check the Poisson property of $\mad$ by a direct computation very similar to the proof of Theorem~\ref{thm:infeq}.

\begin{rem}
	In the language of \cite{Marsden1983}, the map $\mad$ is an example of \textit{symplectic} or \textit{Clebsch variables} for the ``gradient subspace" of the space $\dsalg$, where $(\mu, \rho)\in \dsalg$ with $\mu=\rho v^\flat$ and
	$v=\nabla\tau$. The term ``Clebsch variables'' means a Poisson map $\psi:R\to P$ from a symplectic space $R$ to a Poisson space $P$. Corollary~\ref{poisson} shows that our map $\mad$ is a Poisson map from the symplectic space $\Psi$ (with symplectic form given by $\omega(U,V)=-2\la U,iV \ra$) to the gradient subspace of the Poisson space $\dsalg$. Such a map $\psi$ is also called a \textit{symplectic realization} \cite{wein83}. 
\end{rem}

\begin{rem}
	In \cite{KhesMisMod16, VonRen12}  the Madelung transform is understood somewhat differently. It is a surjection $\sigma$ from the space $\mathcal{C}(M)$ of smooth non-vanishing complex functions on a Riemannian manifold M to the tangent bundle $T\mathcal{P}(M)$ of the space $\mathcal{P}(M)$ of probability measures over $M$. Here $T\mathcal{P}(M)$ plays the role of the phase space for potential motions of the compressible fluid. It is shown in \cite{VonRen12} that $\sigma$ is a symplectic submersion of $\mathcal{C}(M)$ equipped with the constant symplectic structure on complex functions into $T\mathcal{P}(M)$ equipped with the natural symplectic structure induced by the Wasserstein metric.
\end{rem}

\begin{rem}
	One of possible applications of the momentum map nature of the Madelung transform is its use for Hamiltonian reduction to the space of singular solutions such as vortex sheets and vortex membranes. We hope to address this issue in a future publication.
\end{rem}

\end{subsection}

\subsection*{Acknowledgements}
The author is grateful to Boris Khesin for posing the problem of studying the geometric properties of the Madelung transform.

\end{section}

\bibliography{madelung2015}
\bibliographystyle{plain}

\end{document}